\documentclass[12pt]{amsart}     

\usepackage{amsmath,amssymb,amsfonts,amsthm} 
\usepackage{graphicx}                 
\usepackage{enumerate}

\newcommand{\BN}{\mathbb{N}}
\newcommand{\CB}{\mathcal {B}}
\newcommand{\CD}{\mathcal {D}}
\newcommand{\CF}{\mathcal {F}}
\newcommand{\CG}{\mathcal {G}}
\newcommand{\CH}{\mathcal {H}}
\newcommand{\CR}{\mathcal {R}}
\newcommand{\CS}{\mathcal {S}}

\newcommand{\bU}{\bar{U}}
\newcommand{\bV}{\bar{V}}

\newcommand{\Aut}{\mathrm{Aut}}
\newcommand{\FAut}{\mathrm{FAut}}
\newcommand{\Sym}{\mathrm{Sym}}
\newcommand{\FSym}{\mathrm{FSym}}

\newtheorem{defin}{{\bf Definition}}[section]

\newtheorem{obs}[defin]{{\bf Observation}}
\newtheorem{lem}[defin]{{\bf Lemma}}
\newtheorem{prop}[defin]{{\bf Proposition}}

\newtheorem{cor}[defin]{{\bf Corollary}}

\newtheorem{claim}{\noindent {\bf Claim}}

\begin{document}

\title{Overgroups of the Automorphism Group of the Rado Graph}

\author[P.Cameron]{Peter Cameron}
\address{School of Mathematical Sciences, Queen Mary, University of London \\ London E1 4NS, UK}
\email {p.j.cameron@qmul.ac.uk} 
\author[C.Laflamme] {Claude Laflamme*}
\thanks{*Supported by NSERC of Canada Grant \# 690404}
\address{University of Calgary, Department of Mathematics and Statistics \\ Calgary, Alberta, Canada T2N 1N4} 
\email {laflamme@ucalgary.ca} 
\author [M.Pouzet]{Maurice Pouzet}
\address{ICJ, Math\'ematiques, Universit\'e Claude-Bernard Lyon1, 43 Bd. 11 Novembre 1918
F$69622$ Villeurbanne cedex, France and University of Calgary,
Department of Mathematics and Statistics, Calgary, Alberta, Canada T2N
1N4}
\email{pouzet@univ-lyon1.fr }
\author[S.Tarzi]{Sam Tarzi}
\address{School of Mathematical Sciences, Queen Mary, University of London \\ London E1 4NS, UK}
\email{s.tarzi@qmul.ac.uk} 
\author[R.Woodrow] {Robert Woodrow}
\address{University of Calgary, Department of Mathematics and Statistics \\ Calgary, Alberta, Canada T2N 1N4} 
\email {woodrow@ucalgary.ca} 

\date{May 4, 2012}

\begin{abstract}
We are interested in overgroups of the automorphism group of the Rado
graph. One class of such overgroups is completely understood; this is
the class of reducts.  In this article we tie recent work on various
other natural overgroups, in particular establishing group connections
between them and the reducts.
\end{abstract}

\maketitle

\section{Introduction}

The Rado Graph $\CR$ is the countable universal homogeneous graph: it
is the unique (up to isomorphism) countable graph with the defining property
that for every finite disjoint subsets of vertices $A$ and $B$ there
is a vertex adjacent to all vertices in $A$ and not adjacent to any of
the vertices in $B$.

We are interested in overgroups of its automorphism group $\Aut(\CR)$
in $\Sym(\CR)$, the symmetric group on the vertex set of $\CR$.  One
class of overgroups of $\Aut(\CR)$ is completely understood; this is
the class of reducts, or automorphism groups of relational structures
definable from $\CR$ without parameters.  Equivalently, this is the
class of subgroups of $Sym(\CR)$ containing $Aut(\CR)$
which are closed with respect to the product topology. According to a
theorem of Thomas \cite{T1}, there are just five reducts of
$\Aut(\CR)$:
\begin{itemize}
\item $\Aut(\CR)$; 
\item $\CD(\CR)$, the group of dualities (automorphisms and anti-automor\-phisms) of $\CR$; 
\item $\CS(\CR)$, the group of switching automorphisms of $\CR$ (see below);
\item $\CB(\CR) = \CD(\CR).\CS(\CR)$ (the big group);
\item $Sym(\CR)$, the full symmetric group.
\end{itemize}

Given a set $X$ of vertices in a graph $G$, we denote by $\sigma_X(G)$
the {\em switching} operation of changing all adjacencies between $X$
and its complement in $G$, leaving those within or outside $X$
unchanged, thus yielding a new graph.  Now a switching automorphism of
$G$ is an isomorphism which maps $G$ to $\sigma_X(G)$ for some $X$,
and $\CS(G)$ is the group of switching automorphisms. Thus the
interesting question is often for which subset $X$ is $\sigma_X(G)$
isomorphic $G$, and for this reason will sometimes abuse terminology
and may call $\sigma_X(G)$ a switching automorphism. 

Thomas also showed (see \cite{T2}, and also the work of Bodirsky and
Pinsker \cite{BP}) that the group $\CS(\CR)$ can also be understood as
the automorphism group of the 3-regular hypergraph whose edges are
those 3-element subsets containg an odd number of edges. Similarly,
$\CD(\CR)$ is the automorphism group of the 4-regular hypergraph whose
edges are those 4-element subsets containg an odd number of edges, and
$\CB(\CR)$ is the automorphism group of the 5-regular hypergraph whose
edges are those 5-element subsets containg an odd number of edges.

One can see that $\CG$ is any subgroup of $\Sym(\CR)$, then
$\CG.\FSym(\CR)$ (the group generated by the union of $\CG$ and
$\FSym(\CR)$, the group of all finitary permutations on $\CR$) is a
subgroup of $\Sym(\CR)$ containing $\CG$ and highly transitive. The
reducts $\CD(\CR)$ and $\CS(\CR)$ however are 2-transitive but not
3-transitive, while $\CB(\CR)$ is 3-transitive but not
4-transitive. On the other hand we have the following.

\begin{lem} \label{lem:highlytransitive}
Any overgroup of $\Aut(\CR)$ which is not contained in $\CB(\CR)$ is highly
transitive. 
\end{lem}

\begin{proof}
Let $G$ with $\Aut(\CR)\leq G \not\leq \CB(\CR)$, and let $\overline G$ be
the closure of $G$ in $\Sym(\CR)$. Since $\overline G\not \leq
\CB(\CR)$, we have $\overline G= \Sym(\CR)$ by Thomas' theorem. Since
$G$ and $\overline G$ have the same orbits of $n$-uples, $G$ is highly
transitive.

\end{proof}

Now for a bit of notation. With the understanding that $\CR$ is the
only graph under consideration here, we write $v \sim w$ when $v$ and
$w$ are adjacent (in $\CR$), $\CR(v)$ for the set of vertices adjacent
to $v$ (the neighbourhood of $v$), and will use $\CR^c(v)$ for $\CR
\setminus \CR(v)$ (note that $v \in \CR^c(v)$).  We say that a
permutation $g$ changes the adjacency of $v$ and $w$ if $(v \sim w)
\Leftrightarrow (v^g \not \sim w^g)$. We say that $g$ changes finitely many adjacencies at $v$
if there are only finitely many points $w$ for which $g$ changes the
adjacency of $v$ and $w$.

Given two groups $G_1, G_2$ contained in a group $H$, we write
$G_1.G_2$ for the subgroup of $H$ generated by their union. 

In Section \ref{sec:otheroverg}, we present various other natural overgroups
and tie recent work and in particular establish group connections
between them and the reducts.

\section{Other Overgroups of $\Aut(\CR)$}\label{sec:otheroverg}

Cameron and Tarzi in \cite{CT} have studied the following overgroups
of $\CR$.

\begin{enumerate}[a)]
\item $\Aut_1(\CR)$, the group of permutations which change only a 
finite number of adjacencies;

\item $\Aut_2(\CR)$, the group of permutations which change only a
finite number of adjacencies at each vertex;

\item $\Aut_3(\CR)$, the group of permutations which change only a finite 
number of adjacencies at all but finitely many vertices;

\item $\Aut(\CF_\CR)$, where $\CF_\CR$ is the neighbourhood filter of $\CR$, 
the filter generated by the neighbourhoods of vertices of  $\CR$.  
\end{enumerate}

One shows that all these sets of permutations really are groups, as
claimed. For $\Aut_i(\CG)$, this is because if $C(g)$ denotes the set
of pairs $\{v,w\}$ whose adjacency is changed by $g$, then one
verifies that $C(g^{-1}) = C(g)^{g^{-1}}$ and $C(gh) \subseteq C(g)
\cup C(h)^{g^{-1}}$.

The main facts known about these groups are:

\begin{prop}\cite{CT}\label{prop:mainfacts}
\begin{enumerate}[a)]
\item $\Aut(\CR) < \Aut_1(\CR) < \Aut_2(\CR) < \Aut_3(\CR)$; 
\item $\Aut_2(\CR) \leq \Aut(\CF_\CR)$, but $\Aut_3(\CR)$ and
$\Aut(\CF_\CR)$ are {\em incomparable}; 
\item $\FSym(\CR) <  \Aut_3(\CR) \cap \Aut(\CF_\CR)$, but $\FSym(\CR) \cap \Aut_2(\CR) = 1$; 
\item  $\CS(\CR) \not \leq  \Aut(\CF_\CR)$, and $\Aut(\CF_\CR) \cap \CD(\CR) = \Aut(\CF_\CR) \cap \CS(\CR) = \Aut(\CR)$.
\end{enumerate}
\end{prop}

\begin{proof}
(a) is clear. \\

(b) For the first part, let $g \in \Aut_2(\CR)$. It suffices to show
that, for any vertex $v$, we have $\CR(v)^g \in \CF_(\CR)$. Now by
assumption, $\CR(v)^g$ differs only finitely from $\CR(v^g)$; let
$\CR(v)^g \setminus \CR(v^g) = \{w_1, . . . , w_n\}$. If we choose $w$
such that $ w_i \not \in \CR(w)$ for each $i$, then we have
\[ \CR(v^g) \cap \CR(w) \subseteq \CR(v)^g ,\]
and we are done.

For the second part, choose a vertex $v$, and consider the graph
$\CR'$ obtained by changing all adjacencies at $v$. Then $\CR' \cong
\CR$. Choose an isomorphism $g$ from $\CR$ to $\CR'$; since $\CR'$ is
vertex-transitive, we can assume that $g$ fixes $v$. So $g$ maps
$\CR(v)$ to $\CR_1(v) = \CR^c(v) \setminus \{v\} $. Clearly $g
\in \Aut_3(\CR)$, since it changes only one adjacency at any point
different from $v$. But if $g \in \Aut(\CF_\CR)$, then we would have
$\CR_1(v) \in \CF_\CR$, a contradiction since $\CR(v) \cap \CR_1(v) =
\emptyset$.

In the reverse direction, let $\CR''$ be the graph obtained by
changing all adjacencies between non-neighbours of $v$. Again $\CR''
\cong \CR$, and we can pick an isomorphism $g$ from $\CR$ to $\CR''$ which
fixes $v$. Now $g$ changes infinitely many adjacencies at all
non-neighbours of $v$ (and none at $v$ or its neighbours), so $g \not \in \Aut_3(\CR)$. Also, if
$w$ is a non-neighbour of $v$, then $\CR(v) \cap \CR(w)^g = \CR(v)
\cap \CR(w^g)$, so $g \in \Aut(\CF_\CR)$.

(c) Note that any non-identity finitary permutation belongs to $\Aut_3(\CR)
\setminus \Aut_2(\CR)$. For if $g$ moves $v$, then $g$ changes infinitely
many adjacencies at $v$ (namely, all $v$ and $w$, where $w$ is
adjacent to $v$ but not $v^g$ and is not in the support of $g$). On
the other hand, if $g$ fixes $v$, then $g$ changes the adjacency of $v$ and $w$
only if $g$ moves $w$, and there are only finitely many such $w$.  

Finally, if $g \in \FSym(\CR )$, then $\CR(v)^g$ differs only finitely from $\CR(v)$,
for any vertex $v \in V$; so $g \in \Aut(\CF_\CR)$.

Thus the left inclusion is proper: $\Aut_2(\CR)$ is contained in the right-hand side but
intersects $\FSym(\CR)$ in $\{1\}$.

(d) The graph $\CR'$ in the proof of (b) is obtained from $\CR$ by
switching with respect to the set $\{v\}$; so the permutation $g$ belongs to
the group $\CS(\CR)$ of switching automorphisms. Thus $S(R) \not \leq
\Aut(\CF_\CR)$.

Now any anti-automorphism $g$ of $\CR$ maps $\CR(v)$ to a set disjoint
from $\CR(v^g)$; so no anti-automorphism can belong to
$\Aut(\CF_\CR)$.  Suppose that $g \in \Aut(\CF_\CR)$ is an isomorphism
from $\CR$ to $\sigma_X(\CR)$. We may suppose that $\sigma_X$ is not the identity,
that is, $X \neq \emptyset$ and $Y = V \setminus X \neq
\emptyset$. Choose $x$ and $y$ so that $x^g \in X$ and $y^g \in Y$. Then $\CR(x)^g \bigtriangleup 
Y = \CR(x^g)$ and $\CR(y)^g \bigtriangleup X = \CR(y^g)$. Hence $\CR(x^g)
\cap \CR(x)^g \subseteq X$ and $\CR(y^g) \cap \CR(y^g) \subseteq Y$. Hence 
\[ \CR(x^g) \cap \CR(x)^g \cap \CR(y^g) \cap \CR(y)^g = \emptyset,\]
a contradiction.

\end{proof}

On the other hand, results of Laflamme, Pouzet and Sauer in
\cite{LPS} concern the hypergraph $\CH$ on the vertex set of $\CR$
whose edges are those sets of vertices which induce a copy of
$\CR$. Note that a cofinite subset of an edge is an edge. There are
three interesting groups here:
\begin{enumerate}[a)]
\item $\Aut(\CH)$;
\item $\FAut(\CH)$, the set of permutations g with the property that
there is a finite subset $S$ of $\CR$ such that for every edge $E$,
both $(E \setminus S)g$ and $(E \setminus S)g^{-1}$ are edges.
\item $\Aut^*(\CH)$, the set of permutations g with the property that, for every
edge $E$, there is a finite subset $S$ of $E$ such that $(E \setminus
S)g$ and $(E \setminus S)g^{-1}$ are edges.

\end{enumerate}

Clearly $\Aut(\CH) \leq \FAut(\CH) \leq \Aut^*(\CH)$, and a little
thought shows that all three are indeed groups. Moreover one will note
that for $\Aut^*(\CH)$ and $\FAut(\CH)$ to be groups, both conditions
on $g$ and $g^{-1}$ in their definitions are necessary. To see this,
choose an infinite clique $C \subset \CR$, and also partition $\CR$
into two homogeneous edges $E_1$ and $E_2$: for every finite disjoint
subsets of vertices $A$ and $B$ of $\CR$ there is a vertex in $E_2$
adjacent to all vertices in $A$ and not adjacent to any of the
vertices in $B$. Then it is shown in \cite{LPS} that there exists $g
\in \Sym(\CR)$ such that $Cg=A$, and $Eg$ is an edge for any edge
$E$. But clearly $(A \setminus S)g^{-1}$ is not an edge for any
(finite) $S$.

As a further remark let $\CH^*$ be the hypergraph on the vertex set of
$\CR$ whose edges are subset of the form $E\cup F$ where $E$ induces a
copy of $\CR$ and $F$ is a finite subset of $\CR$. Equivalently these
are the subsets of $\CR$ of the form $E\Delta F$ where $E$ induces a
copy of $\CR$ and $F$ is a finite subset of $\CR$ (this follows from the
fact that for every copy $E$ and finite set $F$, $E\setminus F$ is a
copy). Then observe that $Aut(\CH^*)= Aut^*(\CH)$.

\bigskip

We now provide some relationships between these LPS groups and the CT
groups.

\begin{prop}\label{prop:LPSCT}
\begin{enumerate}[a)]
\item $\Aut(\CH) < \FAut(\CH)$.
\item $\Aut_2(\CR) \leq \Aut(\CH)$ and $\Aut_3(\CR) \leq \FAut(\CH)$.
\item $\FSym(\CR) \leq  \FAut(\CH)$ but $\FSym(\CR) \cap \Aut(\CH) = 1$. 
\end{enumerate}
\end{prop}

\begin{proof}
(a) This follows from part (c).  \\ 

(b) If we alter a finite number of adjacencies at any point of $\CR$, the
result is still isomorphic to $\CR$. So induced copies of $\CR$ are preserved
by $\Aut_2(\CR)$.  Similarly, given an element of $\Aut_3(\CR)$, if we throw away
the vertices where infinitely many adjacencies are changed, we are in
the situation of $\Aut_2(\CR)$.  \\

(c) The first part follows from Proposition \ref{prop:mainfacts} part (c)  and part (b) above. For
the second part, choose a vertex $v$ and let $E$ be the set of
neighbours of $v$ in $\CR$ (this set is an edge of $\CH$). Now, for
any finitary permutation, there is a conjugate of it whose support
contains $v$ and is contained in $\{v\}\cup E$. Then $Eg = E \cup
\{v\} \setminus \{w\}$ for some $w$. But the induced subgraph on this
set is not isomorphic to $\CR$, since $v$ is joined to all other
vertices.
\end{proof}

We shall see later that $\Aut(\CH).\FSym(\CR) <\Aut^*(\CH)$, but we
present a bit more information before doing so. In particular we now show
that an arbitrary switching is almost a switching isomorphism.

\begin{lem} \label{lem:switchingcofinite}
Let $X \subseteq \CR$ arbitrary and $\sigma = \sigma_X$ be the
operation of switching $\CR$ with respect to $X$. Then there is a
finite set $S$ such that $\sigma(\CR \setminus S)$ is an edge of
$\CH$, namely isomorphic to the Rado graph.
\end{lem}

\begin{proof}
For $E \subseteq \CR$ and disjoint $U,V \subseteq E$, denote by
$W_E(U,V)$ the collection of all witnesses for $(U,V)$ in $E$. Note
that if $E$ is an edge, then $W_E(U,V)$ is an edge for any such sets
$U$ and $V$. Now for $C \subseteq \CR$, denote for convenience by
$C_X$ the set $C \cap X$, and by $C^c_X$ the set $C \setminus X$.

Thus if $\sigma(\CR)$ is \textbf{not} already an edge of $\CH$, then
the Rado graph criteria regarding switching yields finite disjoint
$U,V \subseteq E$ such that both:
\begin{itemize}
\item $W_\CR(U^c_X \cup V_X, U_X \cup V^c_X) \subseteq X $
\item $W_\CR(U_X \cup V^c_X, U^c_X \cup V_X) \cap X = \emptyset $
\end{itemize}
Define $S=U \cup V$  and $E = \CR \setminus S$, we show that
$\sigma(E)$ is an edge of $\CH$. 

For this let $\bU, \bV \subseteq E$. But now we have:
\begin{eqnarray*}
\lefteqn{W_{\CR}(\bU_X \cup \bV^c_X \cup U^c_X \cup V_X, \bU^c_X \cup \bV_X \cup U_X \cup V^c_X) } \\
&& = W_{E}(\bU_X \cup \bV^c_X,  \bU^c_X \cup \bV_X )  \cap W_{\CR}(U^c_X \cup V_X, U_X \cup V^c_X)  \\
&& \subseteq X 
\end{eqnarray*}

In virtue of the Rado graph, the above first set contains infinitely
many witnesses, and thus $W_{E}(\bU_X \cup \bV^c_X, \bU^c_X \cup
\bV_X)$ is non empty. Hence $\sigma(E)$ contains a witness for
$(U,V)$, and we conclude that $\sigma(E)$ is an edge.
\end{proof}

\bigskip

The last item above shows that any graph obtained from $\CR$ by switching has a
cofinite subset inducing a copy of $\CR$. This can be formulated as
follows: Let $G$ be a graph on the same vertex set as $\CR$ and having the
same parity of the number of edges in any 3-set as $\CR$. Then $G$ has a
cofinite subset inducing $\CR$.

\bigskip

The next result is about the relation between the LPS-groups and the
reducts.

\begin{prop}\label{prop:LPSreducts}
\begin{enumerate}[a)]
\item $\CD(\CR)<  \Aut(\CH)$.
\item $\CS(\CR) \not\leq \Aut(\CH)$.
\item $\CS(\CR) \leq\Aut^*(\CH)$ 
\end{enumerate}
\end{prop}

\begin{proof} 
(a) Clearly $\CD(\CR) \leq \Aut(\CH)$ since $\CR$ is
self-complementary. We get a strict inequality since $\Aut(\CH)$ is
highly transitive (since $\Aut_2(\CR) \leq \Aut(\CH)$) while
$\CD(\CR)$ is not.

(b)We show that $\CR$  can be switched into a graph isomorphic to $\CR$ in such a
way that some induced copy $E$ of $\CR$ has an isolated vertex after switching. Then
the isomorphism is a switching-automorphism but not an automorphism of $\CH$.

Let $p$, $q$ be two vertices of $\CR$. The graph we work with will be
$\CR_1 = \CR \setminus \{p\}$, which is of course isomorphic to $\CR$. Let
$A, B, C, D$  be the sets of vertices joined to $p$ and $q$, $p$ but not $q$, $q$ but
not $p$, and neither $p$ nor $q$, respectively. Let $\sigma$  be the operation of
switching $\CR_1$ with respect to $C$, and let $E = \{q\} \cup B \cup C$. It is clear
that, after the switching $\sigma$, the vertex $q$ is isolated in $E$. So we have
to prove two things: 

\begin{claim}
$E$ induces a copy of $\CR$.  
\end{claim}

\begin{proof}
Take $U$, $V$ to be finite disjoint subsets of $E$. We may assume
without loss of generality that $q \in U \cup V$ .  

{\bf Case 1:} $ q \in U$. Choose a witness $z$ for $(U, V \cup
\{p\})$ in $\CR$. Then $z \not \sim p$ and $z \sim q$, so $z \in C$;
thus $z$ is a witness for $(U,V)$ in $E$.

{\bf Case 2:} $q \in V$ . Now choose a witness for $(U \cup \{p\},V)$
in $\CR$; the argument is similar.  
\end{proof}

\begin{claim}
$\sigma(\CR_1)$ is isomorphic to $\CR$.
\end{claim}

\begin{proof}
Choose $U, V$ finite disjoint subsets of $\CR \setminus \{p\}$. Again,
without loss, $q \in U \cup V$. Set $U_1 = U \cap C$, $U_2 = U
\setminus U_1$, and $V_1 = V \cap C$, $V_2 = V \setminus V_1$. 

{\bf Case 1:} $q \in U$, so $q \in U_2$. Take $z$ to be a witness for
$(U_2 \cup V_1 \cup \{p\}, U_1 \cup V_2)$ in $\CR$. Then $z \sim p, q$,
so $z \in A$. The switching $\sigma$ changes its adjacencies to $U_1$
and $V_1$, so in $\sigma(\CR_1)$ it is a witness for $(U_1 \cup U_2, V_1
\cup V_2)$.  

{\bf Case 2:} $q \in V$, so $q \in V_2$. Now take $z$ to be a witness
for $(U_1 \cup V_2, U_2 \cup V_1 \cup \{p\})$ in $\CR$. Then $z \sim
q$, $z \not\sim p$, so $z \in C$, and $\sigma$ changes its adjacencies
to $U_2$ and $V_2$, making it a witness for $(U_1 \cup U_2, V_1 \cup V_2)$.
\end{proof}

(c) Let $X \subseteq \CR$, $\sigma$ be the operation of switching
$\CR$ with respect to $X$, and $g: \CR \rightarrow \sigma(\CR)$ an
isomorphism.  In order to show that $g \in Aut(\CH^*)$ we need to show
that if $E$ is an edge of $\CH$ there is some finite $S$ such that
$(E\setminus S)g$ and $(E\setminus S)g^{-1}$ are edges of $\CH$.

However the graph $Eg$ (in $\CR$) is obtained from switching the graph
induced by $\sigma(\CR)$ on $Eg$. Since the latter is a copy of
$\mathcal R$, Lemma \ref{lem:switchingcofinite} yields a finite $S_0
\subset \CR$ such that $Eg \setminus S_0$ is an edge of $\CH$. If $S_1= S_0 g^{-1}$, then 
$(E \setminus S_1)g$ is an edge of $\CH$. 

Finally notice that $g^{-1}$ is an isomorphism from $\CR$ to
$\sigma_{Xg^{-1}}(\CR)$, the above argument shows that there is a
finite $S_2$ such that $(E \setminus S_2)g^{-1}$ is an edge of $\CH$.
Since cofinite subsets of edges are edges $S:=S_1\cup S_2$ has the
required property.
\end{proof}

\begin{cor} \label{cor:br}
$\CB(\CR) < \Aut^*(\CH)$
\end{cor}

\begin{proof}
That $\CB(\CR) \leq \Aut^*(\CH)$ follows from parts (a) and (c) of
Proposition \ref{prop:LPSreducts}.

We get a strict inequality since $\Aut^*(\CH)$ is
highly transitive (since $\Aut_2(\CR) \leq \Aut^*(\CH)$) while
$\CD(\CR)$ is not.
\end{proof}

\medskip

\begin{prop}\label{prop:srfaut}
$\CS(\CR) \not \leq \FAut(\CH)$
\end{prop}

In view of $\CS(\CR) \leq \Aut^*(\CH)$ (by Proposition
\ref{prop:LPSreducts}), this yields the following
immediate Corollary.

\begin{cor}\label{cor:autfsym}
$\FAut(\CH) < \Aut^*(\CH)$. 
\end{cor}

Clearly we have the following immediate observation:

\begin{obs} \label{obs:faut}
\[ \Aut(\CH).\FSym(\CR) \leq  \FAut(\CH) \]
\end{obs}

Hence, the Corollary yields yet that $\Aut(\CH).\FSym(\CR) < \Aut^*(\CH)$.

\begin{proof} (of Proposition \ref{prop:srfaut}).
The argument can be thought of as an infinite version of the one given
in part b) of Proposition \ref{prop:LPSreducts}.

We shall recursively define subsets of $\CR$:
\begin{itemize}
\item $A = \langle a_n : n \in \BN \rangle $
\item $B = \langle b_n : n \in \BN \rangle $
\item $C = \langle c_n : n \in \BN \rangle $
\end{itemize}
and set $D$ so that:
\begin{enumerate}
\item $ \forall n \forall k \leq n \; a_n \not\sim b_k \mbox{ and } c_n \sim b_k$
\item $ \forall n \; E_n := \{ a_k : k \geq n\} \cup \{b_n\} \cup \{c_k : k \geq n\}$ is an edge.
\item If $\sigma$ is the operation of switching $\CR$ with respect to $C$, then $\sigma(\CR)$ is isomorphic to $\CR$. 
\item $D = \CR \setminus (A \cup B \cup C)$ is infinite.
\end{enumerate}

\begin{center}
\includegraphics[width=4in]{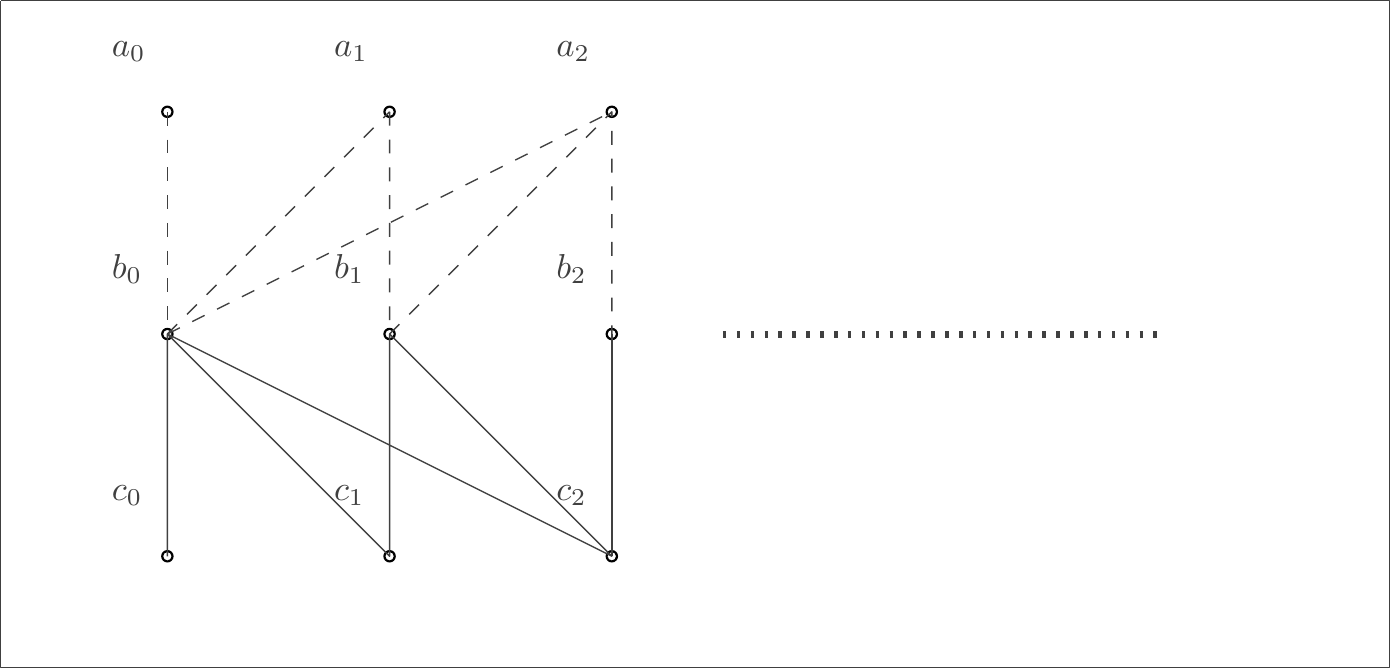}
\end{center}

The construction is as follows. First list all pairs $(U,V)$ of
disjoint finite subsets of $\CR$ so that each one reoccurs infinitely
often. Start with $A=B=C=D= \emptyset$ and at stage $n$, assume we
have constructed $A_n=\{ a_k : k \leq n\}$, $B_n = \{b_k: k \leq n
\}$, and $C_n= \{c_k : k \leq n\}$ satisfying condition 1) above,
together with a finite set $D$ disjoint from $A_n$, $B_n$ and
$C_n$. Then given $(U,V)$, proceed following one of the following
cases:

\begin{enumerate}[a)]
\item Suppose $U \cup V  \subseteq  A_n \cup B_n \cup C_n$ and
contains at most one $b_i$ (i.e. $(U,V)$ is a type candidate for the
eventual $E_i$). Then, choosing from $\CR \setminus ( A_n \cup B_n \cup C_n \cup D)$, add $a_{n+1}$
or $c_{n+1}$ as a witness for $(U,V)$ depending as to whether $b_i$ is
in $V$ or $U$ (add $a_{n+1}$ if there is no such $b_i$ at all).  Then
choose two more elements from $\CR \setminus D$ to complete the
addition of elements $a_{n+1}$, $b_{n+1}$, and $c_{n+1}$ as required
by condition 1). Also throw a new point in $D$ just to ensure it will become infinite.
\item Else add the elements of $U \cup V  \setminus  A_n \cup B_n \cup C_n$ to $D$, 
 and select an element of $\CR \setminus ( A_n \cup B_n \cup C_n \cup
 D)$ as witness to $(U \setminus C_n \cup V \cap C_n, V \setminus C_n \cup U \cap C_n)$. 
\end{enumerate}

The construction in part b) will ensure that $\sigma(\CR)$ is
isomorphic to $\CR$. Indeed let $U$ and $V$ be disjoint finite subsets
of $\CR$, and without loss of generality $U \cap D \neq
\emptyset$. Thus when the pair $(U,V)$ is handled at some stage $n$,
part b) will add a witness $d$ in $D$ to $(U \setminus C_n \cup V \cap
C_n, V \setminus C_n \cup U \cap C_n)$. But then $d$ is a witness to
$(U,V)$ in $\sigma(\CR)$.

Let $g$ be the isomorphism from $\sigma(\CR)$ to $\CR$.  

Finally the construction in part a) clearly ensures condition
2). However note that in $\sigma(\CR)$, $b_n$ is isolated in $E_n$,
and therefore $\sigma(E_n)$ is not an edge. 

Finally for any finite set $S \subset \CR$, choose $n$ so that $S \cap
E_n = \emptyset$. Then $(E_n \setminus S)g$ is not an edge. Thus $g
\in \CS(\CR) \setminus \FAut(\CH)$.

\end{proof}

We now go back to $\Aut(\CF_\CR)$. One can readily verify that the
automorphism $g$ produced in the reverse direction of Proposition
\ref{prop:mainfacts} is in fact not in $\Aut(\CH)$, thus $\Aut(\CF_\CR)\not \leq \Aut(\CH)$. 
However we have the following.

\begin{prop}\label{prop:faut*h}

$\Aut(\CF_\CR)\not \leq \Aut^*(\CH)$. 
\end{prop}

\begin{proof}

Fix a vertex $v \in \CR$. Now partition $\CR^c(v) = A \cup E \cup
D$, where $E$ is an edge, $D$ is an infinite independent set, and $A$ is the set of remaining vertices. This is
easily feasible since $\CR^c(v)$ is an edge. Now define $g \in
\Sym(\CR)$ such that:
\begin{enumerate}[a)]
\item $g \restriction \CR(v)$ is the identity.
\item $g \restriction E$ is a bijection to $D$.
\item $g \restriction A \cup D$ is a bijection to $\CR^c(v)$.
\end{enumerate}
Now for any vertex $w$, $\CR(w)^g \supseteq \CR(v) \cap \CR(w)$ so $g
\in \Aut(\CF_\CR)$. However, for any finite set $S$ of $E$, then $(E
\setminus S)g$ is again an independent set, and thus
certainly not an edge.

Hence $g \not \in \Aut^*(\CH)$ and the proof is complete.
\end{proof}

\section{Conclusion} \label{section:conclusion}

The following diagram summarizes the subgroup relationship between the
various groups under discussion.

\begin{center}
\includegraphics[height=3.5in]{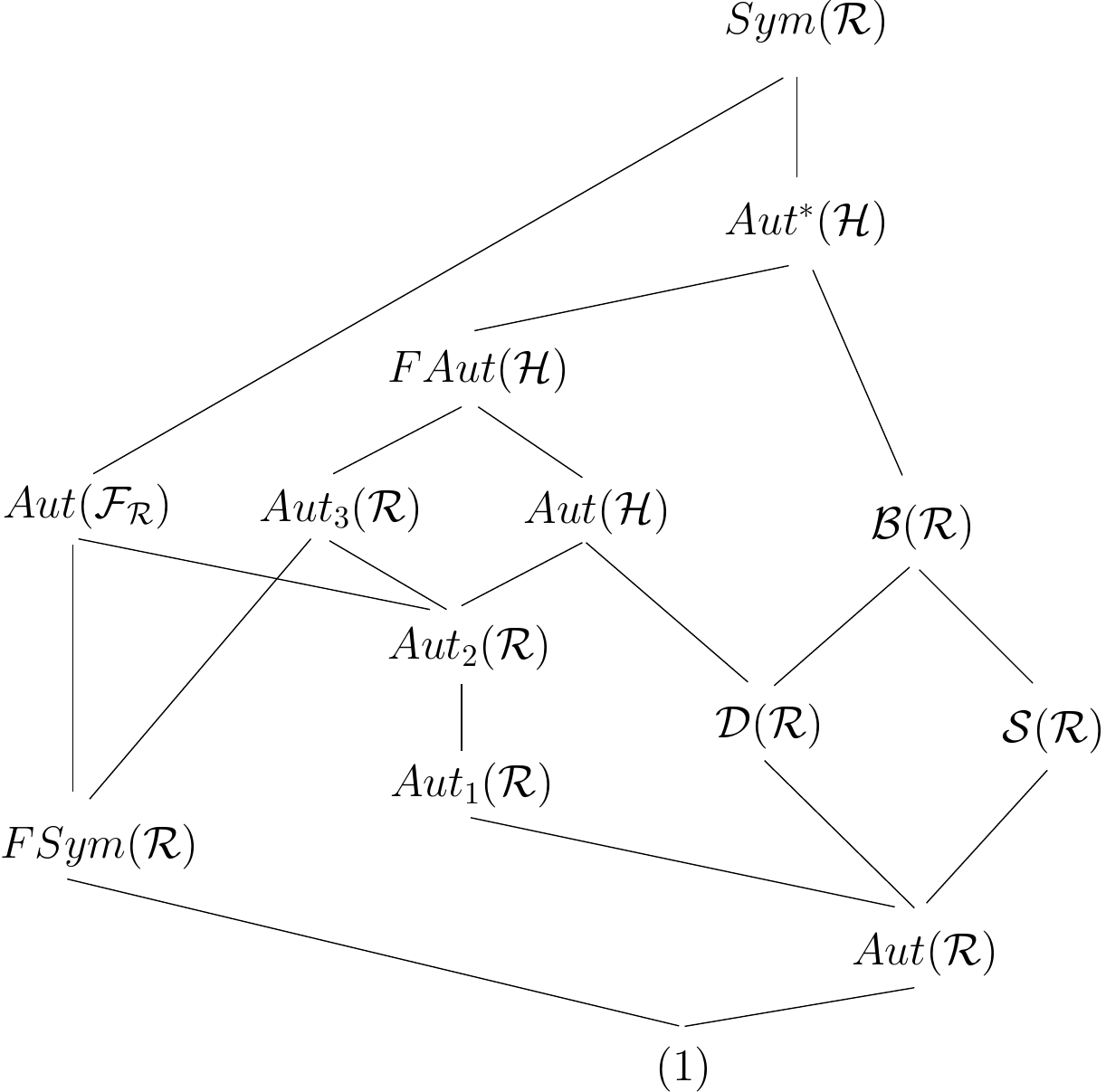}
\end{center}

We do not know if the inclusion is strict in Observation \ref{obs:faut}.

\end{document}